\documentclass[11pt,letterpaper]{amsart}

\usepackage{color}
\usepackage{comment}
\usepackage{amssymb}
\usepackage{amsfonts}
\usepackage{amscd}
\usepackage{slashed}
\usepackage{pdflscape}
\usepackage{tikz}
\usepackage{enumerate}
\usepackage{multirow}
\usepackage{stmaryrd}
\usepackage{ marvosym }

\newcommand{\hh}{{\hspace{.3mm}}}

\newcommand{\pdot}{{\boldsymbol{\cdot}}}

\usepackage{bbold}
\usepackage[backref,pagebackref,linkcolor=blue]{hyperref}
\usepackage{mathrsfs}

\newtheorem{theorem}{Theorem}[section]
\newtheorem{lemma}[theorem]{Lemma}
  
\newtheorem{proposition}[theorem]{Proposition}
\newtheorem{corollary}[theorem]{Corollary}

\theoremstyle{definition}

\theoremstyle{remark}
\newtheorem{remark}[theorem]{Remark}

\usepackage{graphicx} 
\usepackage{amsmath} 
\usepackage{amsfonts}
\usepackage{amssymb}

\newcommand{\be}{\begin{equation}}

\newcommand{\ee}{\end{equation}}

\newcommand{\g}{\go }

\newcommand{\ba}{\begin{array}}

\newcommand{\ea}{\end{array}}

\newcommand{\beq}{\begin{eqnarray}}

\newcommand{\eeq}{\end{eqnarray}}

\newtheorem{lm}{lemma}

\newtheorem{thee}{theorem}

\newtheorem{proo}{proposition}

\newtheorem{co}{corollary}

\newtheorem{rem}{remark}

\newtheorem{deff}{definition}

\newcommand{\bd}{\begin{deff}}

\newcommand{\ed}{\end{deff}}

\newcommand{\bl}{\begin{lm}}

\newcommand{\el}{\end{lm}}

\newcommand{\bp}{\begin{proo}}

\newcommand{\ep}{\end{proo}}

\newcommand{\bt}{\begin{thee}}

\newcommand{\et}{\end{thee}}

\newcommand{\bc}{\begin{co}}

\newcommand{\ec}{\end{co}}

\newcommand{\brm}{\begin{rem}}

\newcommand{\erm}{\end{rem}}

\usepackage{t1enc}

\def\Cal{\mathcal}
\newcommand{\cL}{{\Cal L}}

\newcommand{\newc}{\newcommand}

\let\ccdot\cdot
\def\cdot{\hbox to 2.5pt{\hss$\ccdot$\hss}}

\newc{\aR}{\mbox{\boldmath{$ R$}}}
\newc{\aS}{\mbox{\boldmath{$ S$}}}
\newc{\aT}{\mbox{\boldmath{$ T$}}}
\newc{\aW}{\mbox{\boldmath{$ W$}}}

\newc{\aD}{\mbox{\boldmath{$ D$}}\hspace{-.2mm}}

\newc{\aK}{\mbox{\boldmath{$ K$}}}
\newc{\aL}{\mbox{\boldmath{$ L$}}}

\usepackage{amssymb}
\usepackage{amscd}

\newcommand{\End}{\operatorname{End}}

\newcommand{\cT}{{\mathcal T}}

\let\hash=\sharp  

\newcommand{\nn}[1]{(\ref{#1})}

\newc{\obstrn}[2]{B^{#1}_{#2}}

\newcommand{\rpl}                         
{\mbox{$
\begin{picture}(12.7,8)(-.5,-1)
\put(0,0.2){$+$}
\put(4.2,2.8){\oval(8,8)[r]}
\end{picture}$}}

\newcommand{\lpl}                         
{\mbox{$
\begin{picture}(12.7,8)(-.5,-1)
\put(2,0.2){$+$}
\put(6.2,2.8){\oval(8,8)[l]}
\end{picture}$}}

\usepackage{ifthen}

\newc{\tensor}[1]{#1}
\newc{\Mvariable}[1]{\mbox{#1}}
\newc{\down}[1]{{}_{#1}}
\newc{\up}[1]{{}^{#1}}

\newc{\JulyStrut}{\rule{0mm}{6mm}}
\newc{\midtenPan}{\mbox{\sf S}}
\newc{\midten}{\mbox{\sf T}}
\newc{\midtenEi}{\mbox{\sf U}}
\newc{\ATen}{\mbox{\sf E}}
\newc{\BTen}{\mbox{\sf F}}
\newc{\CTen}{\mbox{\sf G}}

\def\sideremark#1{\ifvmode\leavevmode\fi\vadjust{\vbox to0pt{\vss
 \hbox to 0pt{\hskip\hsize\hskip1em
 \vbox{\hsize3cm\tiny\raggedright\pretolerance10000
 \noindent #1\hfill}\hss}\vbox to8pt{\vfil}\vss}}}%
                        
\numberwithin{equation}{section}

\begin{document}

\renewcommand{\today}{}
\title{Renormalized Yang--Mills Energy  
on  Poincar\'e--Einstein~Manifolds}
\author{A. Rod Gover${}^\heartsuit$, Emanuele Latini${}^\clubsuit$, Andrew Waldron${}^\spadesuit$ \& Yongbing Zhang${}^\diamondsuit$}

\address{${}^\heartsuit$Department of Mathematics\\
  The University of Auckland\\
  Private Bag 92019\\
  Auckland 1\\
  New Zealand} \email{gover@math.auckland.ac.nz}
  
   \address{${}^\clubsuit$
  Dipartimento di Matematica, Universit\`a di Bologna, Piazza di Porta S. Donato 5,
 and  INFN, Sezione di Bologna, Via Irnerio 46, I-40126 Bologna,  Italy}
 \email{emanuele.latini@UniBo.it}
 
  \address{${}^{\spadesuit}$
  	Center for Quantum Mathematics and Physics (QMAP) and\\
  Department of Mathematics\\
  University of California\\
  Davis, CA95616, USA} \email{wally@math.ucdavis.edu}

  \address{${}^\diamondsuit$
School of Mathematical Sciences and Wu Wen-Tsun Key Laboratory of Mathematics \\
University of Science and Technology of China\\ Hefei, Anhui, 230026, P. R. of China.}
\email{ybzhang@amss.ac.cn}

\vspace{10pt}

\vspace{10pt}

\renewcommand{\arraystretch}{1}

\begin{abstract}
We prove that the renormalized Yang--Mills energy on six dimensional Poincar\'e--Einstein spaces 
can be expressed as the bulk integral of a local, 
 pointwise conformally invariant integrand. We show that the latter agrees with the corresponding anomaly boundary integrand in the seven dimensional renormalized Yang--Mills energy. Our methods rely on a generalization of  the Chang--Qing--Yang method for computing renormalized volumes of Poincar\'e--Einstein manifolds, as well as known  scattering theory results for Schr\"odinger operators with short range potentials.

  \vspace{2cm}
\end{abstract}

\maketitle

\pagestyle{myheadings} \markboth{Gover, Latini, Waldron \& Zhang}{Yang-Mills on conformally compact manifolds}

\newpage



\newcommand{\go}{{\mathring {g}}}
\newcommand{\cc}{\boldsymbol{c}}
\newcommand{\DD}{{\sf D}}
\newcommand{\ext}{{\rm d}}
\newcommand{\Trace}{\operatorname{Trace}}

\newc{\LagDen}{\cL_A}
\newc{\hatg}{\hat{\g}}
\newc{\gsup}{\mbox{\textsl{\tiny g}}}
\newc{\hatgsup}{\hat{\mbox{\textsl{\tiny g}}}}
\newc{\Action}{{S}}

%
%
%
%
%
%
%
%
%
%
%
%
%
%
%
%

\section{Introduction}

Poincar\'e--Einstein manifolds and, more generally, conformally compact
Riemannian manifolds, have  proved to be ideal structures for
exposing deep and enlightening relationships between conformal geometry,
scattering theory, functional analysis linked to representation
theory, and the AdS/CFT correspondence in physics~\cite{thebookFG,GrZw,MM,Juhl,GW,SkHe,GrWi}. 
Under certain
circumstances, in this setting divergent energy integrals admit a well-defined finite part called the {\it renormalized energy}.
A first key case is the renormalized Einstein--Hilbert action for Poincar\'e--Einstein manifolds. It computes the renormalized volume of these structures and
 is a fundamental global invariant 
closely linked to the Chern--Gauss--Bonnet invariant \cite{SkHe,GrSrni,Anderson,CQY}. 
Another fundamental case is that of the renormalized Yang--Mills energy functional~\cite{GLWZ}, which is extremely general since it  applies to essentially any connection on the base manifold. 
From its definition it is not at all clear that it can be
recovered as an integrated local invariant; we shall show that on Poincar\'e--Einstein 6-manifolds  it is. We now give some background and then explain how to define this invariant.

%

%
%
%
%

\medskip

The Yang--Mills equations, 
\begin{equation}\label{ymeq}
j[A,g]_a:=g^{bc}\nabla^A_b F^A_{ca}=0\, ,
\end{equation}
are central to modern day particle physics and
geometry~\cite{YM,Donaldson,DonaldsonKron}. In the above ~$g$ is a
(possibly pseudo-)Riemannian metric on a $d$-manifold $M^d$, while
$\nabla^A$ and $F^A$ are the covariant derivative and curvature of a
connection $A$ on a vector bundle $V\to {\mathcal V}M \stackrel \pi
\longrightarrow M$.  The Yang--Mills equations are (generically) the
Euler--Lagrange equations for the Yang--Mills energy functional
  $$
E[A,g]  := -\frac14\int_{M} 
\ext {\rm Vol}(g)\,
\operatorname{Tr}(g^{ac}g^{bd}F^A_{ab}F^A_{cd})\, ,
$$ where the trace is over the composition of endomorphism-valued
two-forms $F^A$. Our
detailed tensor and abstract index conventions are given below.

\medskip

A compact $d$-manifold $M^d$ equipped with a Riemannian metric $g_+$ on its interior $M_+:=M\setminus {\partial M}$ is said to be a {\it Poincar\'e--Einstein manifold} if $g_+$ is Einstein with negative scalar curvature $-d(d-1)$, and there
 exists a smooth defining function $\sigma\in C^\infty M$ 
 such that
 $$ g=\sigma^2 g_+
 $$
 extends smoothly as a metric to the boundary ${\partial M}$. In this case $g$ is
 called a {\em compactified metric}. 
 Any such compactified metric $g$ induces a boundary metric $\bar g$, and in turn a canonical  conformal class of boundary metric $c_{\partial M}$.
  The negative scalar
 curvature Einstein condition implies that the Schouten tensor obeys
\begin{equation}\label{P+}
P^{g_+} = -\frac 12\hh  g_+ \, .
\end{equation}
Many of our results and constructions apply  when the Einstein stricture is dropped, in which case $(M_+,g_+)$ is called a {\it conformally compact manifold}.

The renormalized volume for even dimensional Poincar\'e--Einstein manifolds is constructed by first  defining a {\it cut-off interior} $M_{\varepsilon}:=\{p\in M|\sigma(p)>\varepsilon\}$, where $\varepsilon\in {\mathbb R}_{>0}$. While the volume of $M_+$ with respect to $g_+$ is ill-defined, 
its  {\it regulated volume} ${\rm Vol}_{\varepsilon}:=\int_{M_{\varepsilon } }
\ext {\rm Vol}(g_+)$ has leading asymptotics
$$
{\rm Vol}_{\varepsilon}\stackrel{\varepsilon\to 0}\sim \frac{\varepsilon^{1-d}}{d-1} \int_{\partial M} \frac{\ext {\rm Vol}(\bar g)}{|\ext \sigma|_g}\, .
$$
 The above clearly diverges in the small $\varepsilon$ limit.
The {\it renormalized volume}~${\rm Vol}_{\rm ren}$ is then given by first considering a finite part according to
$$
{\rm Vol}_{\rm ren}^g:=\frac1{(d-1)!}\, \frac{\ext^{d-1}}{\ext \varepsilon^{d-1}}\Big(
\varepsilon^{d-1} \int_{M_{\varepsilon } }
\ext {\rm Vol}(g_+)\Big)\Big|_{\varepsilon=0}\, .
$$
Each boundary  metric~$\bar g\in c_{\partial M}$
uniquely determines (at least to sufficient asymptotic order) a canonical compactified metric~$g^{\rm c}$~\cite{GLee}. In then turns out that ${\rm Vol}_{\rm ren}^{{g^{\rm c}}}=:{\rm Vol}_{\rm ren}$ is  independent of the choice of $\bar g$~\cite{SkHe,GrSrni}. Thus it  is an invariant of the structure $(M,g_+)$. 

\smallskip

An important result of Anderson~\cite{Anderson} is that when $d=4$, the renormalized volume obeys a Chern--Gau\ss--Bonnet type formula
$$
{\rm Vol}_{\rm ren}=\frac{4\pi^2}3 \chi_M -\frac1{24} \int_M \ext {\rm Vol}( g)\hh\big |W^g\big|^2_g\, ,
$$
where $W^g$ denotes the Weyl tensor of $g$ and $\chi_M$ is the Euler characteristic of $M$.
The above integrand is a pointwise conformal invariant, meaning  $$\ext {\rm Vol}(\Omega^2 g)\hh \big|W^{\Omega^2 g} \big|^2_{\Omega^2 g}= \ext {\rm Vol}( g)\hh |W^g|^2_g\, ,$$ for any $0<\Omega\in C^\infty M\, .$

Chang, Qing and Yang have a  proof that the renormalized volume in all even dimensions $d\geq 4$ can  be expressed as a sum of an Euler characteristic term and an integrated pointwise conformal invariant~\cite{CQY}. Their argument relies on several building blocks. These include:  (i) the linear shift by a critical GJMS operator property of the Branson $Q$-curvature under conformal transformations~\cite{BQ,BQ1}; (ii)  factorization of the critical GJMS operator into simpler second order operators~\cite{thebookFG,Go}; (iii) constancy of the Branson $Q$-curvature on Einstein metrics; (iv) the relationship between $Q$-curvature flat metrics and the 
inhomogeneous Laplace equation~\cite{FG02}
$$
\Delta^{g_+} f = 1-d\, ;
$$
(v) the appearance of the renormalized volume as a boundary integral of  Neumann data for solutions to the 
inhomogeneous Laplace equation~\cite{FG02}; (vi) uniqueness of solutions  to the inhomogeneous Laplace equation given  suitable boundary data based on the resolvent for the asympotically hyperbolic Laplacian, constructed by Mazzeo and Melrose~\cite{MM}, and the associated scattering theory and Poisson operator developed by Graham and Zworski~\cite{GrZw};
(vii) a classification of integrated conformal invariants conjectured by the physicists Deser and Schwimmer~\cite{DeSchw}, and proved by Alexakis~\cite{Ale}.

Clearly, the  Chang--Qing--Yang argument 
relies critically on a diverse set of
requisite building blocks. 
While it
has already been applied to the renormalized area of minimal submanifolds of Poincar\'e--Einstein manifolds~\cite{CGKTW}, both
the renormalized volume and submanifold area results hold in a setting where the Chern--Gau\ss--Bonnet formula might have instead been applied along the lines used for Anderson's original  Poincar\'e--Einstein four manifold result.
Our results highlight that the
Chang--Qing--Yang method  can in fact be applied to other geometric settings where no  analog of the Chern--Gau\ss--Bonnet formula applies.

 \smallskip
 
The problem we solve is as follows. Let $(M_+^6,g_+)$ be a Poincar\'e--Einstein manifold and~$A$ a smooth connection on a vector bundle ${\mathcal V}M$ that solves the Yang--Mills equation 
\begin{equation}\label{ymeqs}
j[A,g_+]=0
\end{equation}
on the interior $M_+$. A detailed study of the asymptotics of such solutions, as well as the definition of the accompanying renormalized energy, was given in~\cite{GLWZ}; see also~\cite{Usula,deLima} for related studies of the Yang--Mils equations on  conformally compact manifolds.
We wish to find an
integral formula for 
the {\it renormalized Yang--Mills energy}  defined by
\begin{equation}\label{eren}
E_{\rm ren}:=-\frac14\frac{\ext}{\ext \varepsilon}\Big(
\varepsilon \int_{M_{\varepsilon } }
\ext {\rm Vol}(g_+)
\operatorname{Tr}(g_+^{ab}g_+^{cd}F^A_{ac}F^A_{bd})
\Big)\Big|_{\varepsilon=0}\, ,
\end{equation}
where again, the one parameter family of cut-off interiors $M_{\varepsilon}$ is defined in terms of the metric ${g^{\rm c}}$ determined by $\bar g$.
Note  that because the bulk dimension is even, there is no $\log \varepsilon$ term in  the asymptotics of the {\it regulated Yang--Mills energy}~\cite{GLWZ}
\begin{equation}\label{Ereg}
E_{\varepsilon}:=-\frac14\int_{M_{\varepsilon } }
\ext {\rm Vol}(g_+)
\operatorname{Tr}(g_+^{ab}g_+^{cd}F^A_{ac}F^A_{bd})
\, .
\end{equation}
Partly because of this, the energy in Equation~\nn{eren}
 does not 
depend on the choice of boundary metric $\bar g\in c_{\partial M}$   for~$M$ used to determine the family of cut-off interiors~$M_\varepsilon$~\cite{GLWZ}. It is therefore an invariant of the structure.


\smallskip

 In~\cite{GLWZ} it was further established that the functional gradient of $E_{\rm ren}$ along a path of 
solutions to the Yang--Mills Equations~\nn{ymeqs}
is given by
$$
\hat n^a \hat n^b (\nabla_{\hat n} \nabla_a F_{bc})|_{\partial M} + \bar g^{ab}\big(5 H^g \bar \nabla_a \bar F_{ac}
-(\bar \nabla_b H^g)\bar F_{ac}\big)\, .
$$
The above is a natural conformal invariant. In fact, evaluated on solutions, it  recovers the global  Neumann data left undetermined by  the formal aymptotics
of  the Yang--Mills system on Poincar\'e--Einstein manifolds. In the above display, $\bar \nabla$ denotes the covariant derivative of the connection induced along the boundary by $A$, coupled to the Levi-Civita connection of the boundary metric induced by the compactified metric $g=\sigma^2 g_+$, and~$\hat n$ and $H^g$ are the corresponding  unit normal vector and mean curvature of ${\partial M}$. The curvature of the bundle connection induced along $\partial M$ is denoted by $\bar F$.

\medskip

Our main result is as follows.
\begin{theorem}\label{main}
Let $(M^6_+,g_+)$ be a Poincar\'e--Einstein manifold and $A$  a smooth connection on a vector bundle ${\mathcal V}M$. Moreover suppose that on $M_+$ the Yang--Mills equation 
$$
j[A,g_+]=0
$$
holds,
 $|F|^2_{g_+}$ is nowhere vanishing, and $|\bar F|^2_{\bar g}$ is nowhere vanishing on $\partial M$.
  Then the renormalized Yang--Mills energy is given by
 $$E_{\rm ren}= \frac14\int_{M}  \ext {\rm Vol}(g)
{\mathcal E}_{\rm ren}^g\, ,
$$
where
  ${\mathcal E}_{\rm ren}^g=\Omega^6 {\mathcal E}_{\rm ren}^{\Omega^2 g}$
  is a pointwise conformal invariant,
given for any compactified metric $g =\sigma^2 g_+$ by 
\begin{equation}\label{Eren}
 {\mathcal E}_{\rm ren}\!\!:=\!\!
\operatorname{Tr}\!\Big(\!\!\hh-j^a 
\!\hh
j_a \!
-\tfrac12 (\nabla^c \!\hh F^{ab})\hh\!\nabla_c F_{ab} 
+6 F^{ab} \!\hh P_a{}^c \!\hh F_{bc}\!\hh
-3F^{ab}\nabla_a j_b
\!\hh+
F^{ab} W^c{}_{ab}{}^d F_{dc}
+2F^{ab}\!\hh  F^{c}{}_a F_{bc}\!\Big)\!\hh .
\end{equation}
Here $j_a:=j_a[A,g]$ and $0<\Omega\in C^\infty M$.
\end{theorem} 

\noindent In the above
$|F|^2_{g}:=-\operatorname{Tr}(g^{ac}g^{bd}F^A_{ab}F_{cd}^A)$. Although
we assume, in the theorem, that~$|F|^2_{g_+}$ is nowhere vanishing,
this is a technical requirement imposed by our proof, that may
possibly be relaxed. Of course for many classes of connections,
$|F|^2_{g}$ is only zero where the connection is flat.
Our assumption also imply that $|F|_{g_+}$ is not square integrable with respect to the measure $ \ext {\rm Vol}(g_+)$. Note that if it were to be square integrable, no energy renormalization is required~\cite{GLWZ}.

Our proof method and the structure of the article are summarized below.
The first ingredient we need to establish is a Yang--Mills analog of the Branson $Q$-curvature. Such is given by the Branson--Gover $Q$-curvature of Equation~\nn{Q} which satisfies the linear shift property
$$
\Omega^6Q[\Omega^2 g,A]-Q[g,A]=-3\operatorname{Tr} \nabla^c\big(F^{ab} F_{[ab}\nabla_{c]} \log \Omega \big)\, .
$$
While the Branson--Gover $Q$-curvature supplies an integrated conformal invariant on closed 6-manifolds, we also need a pointwise conformal invariant that differs from the $Q$-curvature by a total divergence. Exactly such appeared as the boundary integrand for the  anomaly in the renormalized Yang--Mills energy on Poincar\'e--Einstein 7-manifolds in~\cite{GLWZ}. This is also presented in Section~\ref{energies}; see Lemmas~\ref{pointE} and~\ref{E2Q}. Indeed this  anomaly integrand exactly reproduces the six dimensional renormalized energy integrand ${\mathcal E}_{\rm ren}$ of Theorem~\ref{main}; see Lemma~\ref{Q2ren}.

The next key step is to establish that there exists a sufficiently smooth conformal rescaling $\tilde g=\tilde \sigma^2 g_+$ of the Poincar\'e--Einstein metric such that the integral of $Q[A,\tilde g]$ vanishes. This is achieved by first showing that the function $f=\tilde \sigma/\sigma$ can be determined by solving an inhomogeneous Schr\"odinger equation. This is achieved in Section~\ref{ise}, where it is also shown that $Q[g_+,A]=|F|^2_{g_+}$ when $g_+$ is a Poincar\'e--Einstein metric and $A$ is a Yang--Mills connection. This is the quantity whose integral we wish to renormalize.

It is critical that the inhomogeneous Schr\"odinger equation has a unique solution for suitable boundary conditions. Here, we are able to mainly rely on  earlier work of Mazzeo--Melrose~\cite{MM}, Graham--Zworski~\cite{GrZw} and Joshi--S\'a Barreto~\cite{BaJo}. This is dealt with in Section~\ref{scatter}.

The last key ingredient necessary before assembling the above building blocks is to adapt an argument of Fefferman and Graham to the inhomogeneous Schr\"odinger equation to show that the Neumann data integrated over the boundary recovers the renormalized energy. This, as well as the final assembly of ingredients is presented in Section~\ref{renorm}.

The article concludes with two examples in Section~\ref{exes}, where we consider the cases where (i) $A$ is the tractor connection of the conformal class of metrics determined on $M$ by~$g_+$, and (ii) where 
$F$ is a closed and coclosed two form and therefore solves the Maxwell equations special case of the Yang--Mills system. The former example determines the renormalization of the energy integral whose integrand is the squared Weyl curvature. This recovers
 the result for such due to Case {\it et al}~\cite{Caseetal}
 and  involves a non-trivial six dimensional conformal invariant due to Fefferman and Graham~\cite{FG-ast}.
%
%

%
%
%
%
%
%
%
%
%
%
%
%

\subsection{Conventions}

We denote by $M^d$ a smooth, oriented $(d\geq3)$-manifold, and drop the superscript $d$ when context makes this clear. By smooth we mean $C^\infty$. The exterior derivative will be denoted by $\ext$.
Tensors on  $M$ will  often be handled using abstract index notation~\cite{Penrose}, so that a metric tensor $g\in \Gamma(\odot^2 T^*M)$  is then denoted $g_{ab}$ and its inverse by $g^{ab}$, thus $g^{ac}g_{cb}=:\delta_b^a$ is the identity endomorphism on $\Gamma(TM)$. Indices are raised and lowered in the standard way,  for example $v^a := g^{ab} v_b\in \Gamma(T^*M)$ when $v_a\in \Gamma(T^*M)$,
and when  multiple metrics are present, we often write $v^a_g$ to avoid confusion. Throughout, we take all metrics to be positive definite.

 We also employ  shorthand notations such as
$v^2:=g_{ab} v^a v^b$, $u.v=u_a v^a$.
Square brackets denote antisymmetrization (with unit weight) over a group of indices, so for example
$
X_{[ab}Y_{c]}:=\frac16(X_{ab}Y_c +
X_{bc}Y_a +
X_{ca}Y_b -
X_{ba}Y_c -
X_{ac}Y_b -
X_{cb}Y_a)
$.

The  volume element defined by a metric $g$ is denoted~$\ext {\rm Vol}(g)$ and its 
Levi-Civita by~$\nabla^g$ (or simply $\nabla$ if context allows).
 The Riemann curvature tensor $R^g$ is defined, in the abstract index notation, by
$$
[\nabla_a , \nabla_b] v^c := R_{ab}{}^c{}_d v^d\, ,
$$
where $v\in \Gamma(TM)$. We will also denote the above by the operator $R_{ab}{}^\#$ acting on $v^c$,
which acts on higher rank tensors in the obvious way. 
The Ricci tensor $Ric_{ab}:=-R_{ac}{}^c{}_b$. The Schouten tensor $P$ and its trace~$J:=P_a{}^a$ are defined by
\begin{equation*}
Ric=:(d-2) P + g J\, ,
\end{equation*}
while the trace-free Weyl tensor $W$ is then determined by  
\begin{equation*}
R_{abcd}=W_{abcd}+g_{ac} P_{bd}
-g_{bc} P_{ad}
+g_{bd} P_{ac}
-g_{ad} P_{bc}\, .
\end{equation*}
The
Cotton tensor
$C_{abc}:= \nabla_a P_{bc}-\nabla_b P_{ac}$
obeys the identity
\begin{equation}\label{divWeyl}
\nabla^a W_{abcd}
=(d-3) C_{cdb}\, .
\end{equation}
The rough Laplacian $\Delta^g:=g^{ab} \nabla_a \nabla_b$ and acts on tensors of any type. 

Given a vector bundle ${\mathcal V}M$ over $M$,  compositions 
of endomorphisms~$X,Y$ are denoted  by juxtaposition $XY$,
and  the trace of such by  $\operatorname{Tr} XY$. We write  $\nabla^A$ for a connection~$A$ on~${\mathcal V}M$, and use the notation $\nabla^g$   when this is the  Levi-Civita connection of a metric. When context allows, we may even write $\nabla$ for either of these or their coupling.
   The curvature $F^A$ (or simply $F$ for short) of $\nabla^A$ is the endomorphism-valued two-form, defined by a commutator  of connections acting on sections of ${\mathcal V}M$
$$
[\nabla^A_a,\nabla^A_b]=F^A_{ab}\, .
$$
Note that if $v^a\in \Gamma(TM\otimes{\mathcal VM})$, then
$$
[\nabla_a,\nabla_b] v^c = (\nabla_a\nabla_b -\nabla_b\nabla_a) v^c = R_{ab}{}^c{}_d v^d + F_{ab} v^c\, .
$$

We employ an ${\mathcal O}$ notation, where $X={\mathcal O}(f^k)$ means that the tensor field $X=f^k Y$ for some other smooth tensor $Y$. 
Complex conjugation is denoted by~$*$.

\section{Energy Formul\ae}\label{energies}

To describe candidates for the renormalized Yang--Mills energy, we first seek integrated conformal invariants of bundle connections on six-manifolds. 
One such candidate
was constructed using ``holographic'' methods in~\cite{GLWZ}, where it arose as the anomaly of the renormalized
 Yang--Mills energy on seven-manifolds.
 It measures the failure of the latter to be conformally invariant and is constructed from a local, pointwise conformally invariant energy integrand.
 \begin{lemma}[Theorem 1.9~\cite{GLWZ}]\label{pointE}
Let $(M,g)$ be a Riemannian $6$-manifold, possibly with boundary, and $A$ a connection on a vector bundle ${\mathcal V}M$. Then
$$
{\operatorname{En}} [g,A]:=
\frac14\int_M 
\ext {\rm Vol}(g)\,{\mathcal A}[g,A]
$$
is conformally invariant,  $\operatorname{En}[\Omega^2g,A]=\operatorname{En}[g,A]\, ,$ 
where
\begin{equation}\label{calE}
{\mathcal A}[g,A]:=\operatorname{Tr}\Big(
-\frac14 \Delta(F^{ab}F_{ab})
- j^a j_a
-2F^{ab} \nabla_a j_b 
+4 F^{ab} P_a{}^c F_{bc} + J F^{ab}F_{ab}
\Big)\
\end{equation}
is conformally invariant; here meaning $\Omega^6{\mathcal A}[\Omega^2g,A]=\mathcal{A}[g,A]$, for all
 $0<\Omega\in C^\infty M$.
\end{lemma}

\begin{remark}
The method used in~\cite{GLWZ} to construct ${\mathcal A}$ in
\eqref{calE} also treats other dimensions and gives such invariant
integrands of derivative order $d-2$ on even dimensional
Poincar\'e--Einstein $d$-manifolds.
\end{remark}

Another  invariant energy is instead given in terms of a
$Q$-curvature-like quantity, due to Branson and
Gover~\cite{BG-Pont,GoPeSl}. This quantity will be key to constructing our
renormalized Yang--Mills energy.
It  relies on the operator defined in the following lemma.
\begin{lemma}
Let $(M^6,g)$ be  a Riemannian manifold and   $A$ a connection on a vector bundle~${\mathcal V}M$. Moreover suppose $X\in \Gamma(\wedge^2T^*M\otimes \End{\mathcal V}M)$ is $\nabla^A$-closed,  then
$$Q_2 X_{ab}:= \nabla_{[a} \nabla^{c} X_{b]c} 
+4 P_{[a}{}^c X_{b]c}+  J X_{ab}
$$
obeys
$$
\Omega^{2}\big[Q_2 X_{ab}\big]_{\Omega^2 g} =Q_2X_{ab}
-3\nabla^c \nabla_{[c} \big(X_{ab]}\log\Omega\big)\, ,
$$
where $\Upsilon:= \ext \log \Omega$, for any $0<\Omega\in C^\infty M$.
\end{lemma}

\begin{proof}
This result was actually established in~\cite{BG-Pont,GoPeSl}, but we give a proof for completeness.
First note that the closedness conditions 
$\nabla_{[a} X_{bc]}=0$ 
and  $\nabla_{[a}\Upsilon_{b]}=0$,
 imply
\begin{align*}
\nabla^c \nabla_{[c} \big(X_{ab]}\log\Omega\big)&=
 \nabla^c \big(\Upsilon_{[c} X_{ab]})
 \\ &=\tfrac13 (\nabla.\Upsilon) X_{ab}+\tfrac13 \nabla_\Upsilon X_{ab}
 +\tfrac23 (\nabla_{[a} \Upsilon^c)X_{b]c}\
 +\tfrac23 \Upsilon_{[a} \nabla^cX_{b]c}\, .
\end{align*}
On the other hand, direct computations (in  $d$ dimensions) show that
\begin{eqnarray*}
\Omega^{{2}}\big[\nabla_{[a} \nabla^c 
\!\!\!\!\!&\!\!\!\!\!&\!\!\!\!\!\!\!\!\!\!\!\!
X_{b]c}\big]_{\Omega^2 g} -\, \nabla_{[a} \nabla^c X_{b]c}
\\
&=&
-2\Upsilon_{[a}\nabla^c X_{b]}{}_c
+g^{cd}\nabla_{[a}( \Upsilon^\hash_{|c|} X_{b]d})
-2g^{cd}\Upsilon_{[a} \Upsilon^\hash_{|c|} X_{b]d}\\
&=&-2 \Upsilon_{[a} \nabla^c X_{b]c}
-\tfrac{d-4}2\nabla_\Upsilon X_{ab}
+(d-4)\big[(\nabla_{[a}\Upsilon^c) X_{b]c}
-2\Upsilon_{[a}X_{b]\Upsilon}\big]\, ,
\end{eqnarray*}
while
\begin{eqnarray*}
\Omega^{2}\big[P_{[a}{}^c X_{b]c}\big]_{\Omega^2 g}
-P_{[a}{}^c X_{b]c}
&=&
-(\nabla_{[a}\Upsilon^c) X_{b]c} + \Upsilon_{[a} X_{b]\Upsilon}+\tfrac12 \Upsilon^2 X_{ab}
\, ,\\[1mm]
\Omega^{2}\big[J X_{ab}\big]_{\Omega^2 g}-J X_{ab}\quad\:&=&
-(\nabla.\Upsilon) X_{ab}-\tfrac{d-2}2\Upsilon^2 X_{ab} 
\, .
\end{eqnarray*}
In the above we used that 
the contorsion $K$ for the difference of two connections $\nabla^{\Omega^2 g}-\nabla^g $ is given by $K_{ab}{}^c = \Upsilon_a \delta_b{}^c + \Upsilon_b \delta_a^c - g_{ab} \Upsilon_g^c$. We have denoted by~$\Upsilon^\hash$ the operator $\nabla^{\Omega^2 g}-\nabla^g $, and used that
$\Upsilon^\hash_a g^{cd}=2 \Upsilon_a g^{cd}$
while $\Upsilon^\hash_c V^c=d \hh \Upsilon_c V^c$.
Setting $d=6$, the result now follows.
\end{proof}

The {\it Branson--Gover $Q$-curvature}
 of~\cite[Theorem 5.3]{BG-Pont}  (developed further in~\cite{GoPeSl}) is then defined by
\begin{equation}\label{Q}
Q[g,A]:=\operatorname{Tr}\big(
F^{ab}Q_2 F_{ab}\big)
=
\operatorname{Tr}\big(- F^{ab} \nabla_a j_b
+4 F^{ab} P_a{}^c F_{bc}+ J F^{ab}F_{ab}\big)\\
\, .
\end{equation}
By virtue of the above lemma, it obeys the linear shift property
\begin{equation}\label{Qshift}
\Omega^{6} Q[\Omega^2 g,A]=Q[g,A]-3
\operatorname{Tr} F^{ab} \nabla^c \nabla_{[c} (F_{ab]}\log \Omega)\, ,
\end{equation}
for any $0<\Omega\in C^\infty M$. The second term on the right hand side is in fact a total divergence because $F$ is identically $\nabla^A$-closed. Hence
$$
{\operatorname{En}}^\emptyset[g,A]:=
\frac14\int_M 
\ext {\rm Vol}(g)\,Q[g,A]
$$
is conformally invariant on closed manifolds, as exploited in~\cite{GoPeSl,GLWZ}.  
In fact it is linked to the energy of Lemma~\ref{pointE}.

\begin{lemma}\label{E2Q}
Let $(M,g)$ be a closed Riemannian $6$-manifold and $A$ a connection on a vector bundle ${\mathcal V}M$. Then
\begin{equation}\label{anidentity}
{\mathcal A}[g,A]= Q[g,A] +
\operatorname{Tr}\Big(-\frac14 \Delta(F^{ab}F_{ab})
- \nabla_a(F^{ab} j_b)\Big)\, .
\end{equation}
Thus,
 if $\partial M=\emptyset$,   
 $$
{\operatorname{En}} [g,A]= {\operatorname{En}}^\emptyset [g,A]\, .
$$
\end{lemma}

\begin{proof}
From Equations~\nn{Q} and~\nn{calE} we have
\begin{align*}
Q+ \operatorname{Tr} F^{ab} \nabla_a j_b
&=
 \operatorname{Tr}\big(4 F^{ab} P_a{}^c F_{bc}+ J F^{ab}F_{ab}\big)\\
&  ={\mathcal A} +   \operatorname{Tr} \big(
 2F^{ab} \nabla_a j_b + j^a j_a +\frac14\Delta(F^{ab}F_{ab}) \big)\, .
\end{align*}
The result now follows upon using that
$$
 \operatorname{Tr} \nabla_a(F^{ab} j_b)
= \operatorname{Tr} \big(
j^aj_a + F^{ab} \nabla_a j_b
\big)\, .
$$
\end{proof}

There are two further obvious pointwise conformal invariants, namely $\operatorname{Tr}
F^{ab} W^c{}_{ab}{}^d F_{dc}$ and $\operatorname{Tr} F^{ab} F^{c}{}_a
F_{bc}$, that can be integrated invariantly over a six-manifold.  A
combination of these obey a Weitzenb\"ock-type formula.
\begin{lemma}\label{WBL}
Let  $A$ be a vector bundle connection on $(M^6,g)$. Then
\begin{multline}\label{WB}
F^{ab} W^c{}_{ab}{}^d F_{dc}
+2 F^{ab} F^{c}{}_a F_{bc}
=\\
-\frac14\Delta(F^{ab}F_{ab})
+\frac12(\nabla^c F^{ab})\nabla_c F_{ab}
+F^{ab}  \nabla_a j_{b}
-2F^{ab} P_a{}^cF_{bc}+
JF^{ab} F_{ab}\, .
\end{multline}

\end{lemma}

\begin{proof}
We compute  directly as follows..
\begin{align*}
\frac12\Delta(F^{ab}F_{ab})
&-(\nabla^c F^{ab})\nabla_c F_{ab}
= F^{ab} \Delta F_{ab}
=-2F^{ab} \nabla^c \nabla_a F_{bc}\\
&=2F^{ab}  \nabla_a j_{b}
-2F^{ab} R^c{}_{a}{}^\# F_{bc}
-2F^{ab} [F^c{}_a ,F_{bc}]
\\
&=2F^{ab}  \nabla_a j_{b}
-2F^{ab} R^c{}_{ab}{}^d F_{dc}
-2F^{ab} R^c{}_{ac}{}^d F_{bd}
-4 F^{ab} F^{c}{}_a F_{bc}
\\
&=2F^{ab}  \nabla_a j_{b}
-2F^{ab} W^c{}_{ab}{}^d F_{dc}
-4F^{ab} P_a{}^cF_{bc}
+2JF^{ab} F_{ab}
-4 F^{ab} F^{c}{}_a F_{bc}
\, .
\end{align*}
\end{proof}

We can use the above identity to show that the pointwise invariant
${\mathcal A}$ equals the renormalized energy integrand
${\mathcal E}_{\rm ren}$ of Theorem~\ref{main}.

\begin{lemma}\label{Q2ren}
Let $(M,g)$ be a  Riemannian $6$-manifold, possibly with boundary,  and $A$ a connection on a vector bundle ${\mathcal V}M$.
 Then 
$$
{\mathcal A} = {\mathcal E}_{\rm ren} \, .
$$
\end{lemma}
\begin{proof}
Employ the Weitzenb\"ock formula~\nn{WB} to eliminate the term $-\frac14 \operatorname{Tr}
 \Delta(F^{ab}F_{ab})$ in Equation~\nn{calE}.
\end{proof}

\section{The Inhomogeneous Schr\"odinger Equation}\label{ise}

We are interested in finding  a defining function $\tilde \sigma$ such that the compactified metric $$\tilde g:=\tilde\sigma^2 g_+$$  is  $Q$-curvature flat, or at least has zero total $Q$-curvature. We first we need a suitable equation for $\tilde \sigma$; this is the aim of the current section. 

\smallskip

We rely on
 two properties of the  $Q$-curvature peculiar to Yang--Mills connections on Poincar\'e--Einstein manifolds. 
\begin{lemma}\label{thenewone}
Let $(M_+,g_+)$ be a six dimensional  Poincar\'e--Einstein manifold and $A$
 be a smooth vector bundle connection on $M$ solving the Yang--Mills equation 
$
j[A,g_+]=0
$
on the interior $M_+$.
Then 
$$
Q[g_+,A]=|F|^2_{g_+}\, .
$$
\end{lemma}

\begin{proof}
Apply $j[A,g_+]=0$ and Equation~\nn{P+} to Equation~\nn{Q}.
\end{proof}

\begin{lemma}\label{shift}
Let $(M_+,g_+)$ be a six dimensional  Poincar\'e--Einstein manifold and 
$A$ be a  smooth vector bundle connection on $M$ solving the Yang--Mills equation 
$
j[A,g_+]=0
$
on the interior $M_+$. Then, for any $f\in C^\infty M_+$,  
\begin{equation}\label{Qshifty}
e^{6f}Q[e^{2f}g_+,A]=|F|_{g_+}^2 +  |F|_{g_+}^2 \Delta^{g_+} f + (\nabla^a_{g_+} |F|^2_{g_+})\nabla_a^{g_+} f
-2 \operatorname{Tr}\big(\nabla^c_{g_+}(F^{ab}_{g_+} F_{ca}  \nabla_{b}f)\big)\, .
\end{equation}
\end{lemma}
\begin{proof}
Here all index manipulations and connections are those of the
Poincar\'e--Einstein metric $g_+$, and we rely on the Yang--Mills
equation $\nabla^a_{g_{+}}F_{ab}=0$.
Using Equation~\nn{Qshift} we compute as follows.
\begin{align*}
e^{6f}Q[e^{2f}g_+,A]\!\!&-
Q[g_+,A]=
-3  \operatorname{Tr}\big(F^{ab}
\nabla^c \nabla_{[c}
(F_{ab]}  f)\big)
=-3  \operatorname{Tr}\big(F^{ab} \nabla^c (F_{[ab} \nabla_{c]} f)\big)\\
&=  |F|^2 \Delta f + \frac12(\nabla^c |F|^2)\nabla_c f
-2  \operatorname{Tr}\big(F^{ab} F_{ca} \nabla^c \nabla_{b}f\big)\\
&=   |F|^2 \Delta f +\frac12 (\nabla^c |F|^2)\nabla_c f
+2 \operatorname{Tr}\big((\nabla^c F^{ab}) F_{ca}\nabla_b f
- \nabla^c(F^{ab} F_{ca}  \nabla_{b}f)\big)\\
&=   |F|^2 \Delta f + \frac12(\nabla^c |F|^2)\nabla_c f
-\operatorname{Tr}\big((\nabla^b F^{ca}) F_{ca}\nabla_b f
+2 \nabla^c(F^{ab} F_{ca}  \nabla_{b}f)\big)\\
&=   |F|^2 \Delta f + (\nabla^c |F|^2)\nabla_c f
-2 \operatorname{Tr} \nabla^c(F^{ab} F_{ca}  \nabla_{b}f)\, .
\end{align*}
The fourth line above employed the Bianchi identity $\nabla^{[c} F^{ab]}=0$. The result then follows from Lemma~\nn{thenewone}.
\end{proof}

The last term on the right hand side of Equation~\nn{Qshifty} is a total divergence that, it will turn out, does not contribute to the total $Q$-curvature for suitable $\tilde g$. Hence we wish to solve
\begin{equation}
\label{farfarabove}
  |F|_{g_+}^2 \Delta^{g_+} f + (\nabla_{g_+}^a |F|^2_{g_+})\nabla^{g_+}_a f+|F|_{g_+}^2=0\, .
\end{equation}
To treat this we henceforth assume that $|F|_{g_+}$ is nowhere vanishing.
 Let us now consider the following modification of the Laplacian of $g_+$ acting on functions, 
\begin{equation}\label{selfadjoint}
\Delta^{g_+,A} := \frac{1}{\: |F|_{g_+}} \nabla_a^{g_+} \circ g_+^{ab} |F|_{g_+}^2 
\nabla_b^{g_+} \circ \frac{1}{\:|F|_{g_+}}\, .
\end{equation}
By construction, $\Delta^{g_+,A} $ is formally self adjoint with respect to the volume form $\ext {\rm Vol}(g_+)
$. This allows us to rewrite Equation~\nn{farfarabove} in a more tractable form, namely that of an inhomogeneous Schr\"odinger equation.
\begin{lemma}
Let $(M_+,g_+)$ be a six dimensional  Poincar\'e--Einstein manifold and 
$A$ be a smooth vector bundle  connection on $M$. Moreover suppose that on $M_+$ we have that $|F|_{g_+}$ is nowhere vanishing and the Yang--Mills equation 
$
j[A,g_+]=0
$
holds.
 If $f^{{\prime}}\in C^\infty M_+$ obeys \label{add pluses to indicate $g_+$ metric, inhomogeneous Schr\"odinger}
\begin{equation}\label{FROGGO}
\Delta^{g_+,A} f^{\prime} = -|F|_{g_+}\, ,
\end{equation}
then 
\begin{equation}\label{nearlyzero}
e^{6f}Q[e^{2f}g_+,A]=-2 \nabla^c_{g_+}\!
\operatorname{Tr}(F^{ab}_{g_+} F_{ca}  \nabla_{b}f)\, ,
\end{equation}
where $f=f^{\prime}/|F|_{g_+} $.
\end{lemma}

\begin{proof}
By Lemma~\ref{shift},
requiring that $f$ obeys Equation~\nn{farfarabove}
implies the desired Equation~\nn{nearlyzero}.
But we can rewrite the operator $\Delta^{g_+,A}$ as
$$
\Delta^{g_+,A} = \Big(|F|_{g_+}
\Delta^{g_+}
 + \frac{1}{\: |F|_{g_+}} 
 (\nabla^a_{g_+} |F|^2_{g_+})\nabla_a^{g_+}
 \Big)\circ  \frac{1}{\: |F|_{g_+}} 
 \, ,
$$
so Equation~\nn{farfarabove} for $f$ becomes
$$
 \Delta^{g_+,A}\Big( |F|_{g_+}f\Big) = -|F|_{g_+} \, .
$$
Setting  $f^\prime=|F|_{g_+}f$ gives the stated result.
\end{proof}

In fact, Equation~\nn{FROGGO} is closely related to problems already addressed in the literature~\cite{MM,GrZw,BaJo}.
A simple computation establishes that 
 \begin{equation}\label{thecat}
\Delta^{g_+,A}=
\Delta^{g_+} -\hh\frac{(\Delta_{g_+}|F|_{g_+})}{|F|_{g_+}}\, .
\end{equation}
The first term is the well-studied Laplace operator on a Poincar\'e--Einstein manifold while the second can be adjusted (in any dimension $d$) by a constant to yield a bounded short-range potential, as shown in the following. 
\begin{lemma}\label{thepot} 
  Let $(M_+,g_+)$ be a Poincar\'e--Einstein manifold and $A$ be a smooth vector bundle connection on $M^d$ such that the Yang--Mills equation, 
$
j[A,g_+]=0
$,
holds on $M_+$.
Assuming that 
  $|F|_{g_+}$ is nowhere zero on $M_+$ and  $|\bar F|^2_{\bar g}$ is nowhere zero on $\partial M$, then  the potential
$$
V:=\frac{(\Delta_{g_+}|F|_{g_+})}{|F|_{g_+}}+2(d-3)
$$
extends to a smooth function on $M$
such that
$$
V|_{\partial M} =0\, .
$$
\end{lemma}

\begin{proof}
Let $\sigma$ be a defining function for $\partial M$. We may write 
$$
|F|_{g_+}=\sigma^2 \Phi\, ,
$$
where $\Phi:= |F|_g\in C^\infty M$ (since the underlying connection $A$ is smooth on $M$) and $g=\sigma^2 g_+$ is a compactified metric.
A short computation relates the scalar Laplacians of~$g$ and $g_+$ acting on $|F|_{g_+}$ along $M_+$:
\begin{align*}
\Delta^{\sigma^{-2}g}(\sigma^2 \Phi)&=
\sigma^2 (\Delta^g -(d-2) \sigma^{-1} \nabla_n) (\sigma^2 \Phi)\\
&=\sigma^2 \Big(\sigma^2\Delta^g -(d-6)\sigma \nabla_n -2 (d-3) |n|_g^2 +2\sigma\nabla.n
\Big)\Phi\, ,
\end{align*}
where $n:=\ext \sigma$. But, on a Poincar\'e--Einstein manifold, any
defining function-compactified metric pair $(g,\sigma)$ obeys~\cite{Goal}
\begin{equation}\label{wehavegoals}
|n|_g^2=1 -2 \rho \sigma\, ,
\end{equation}
where $\rho  = -\frac1d(\nabla.n + \sigma J^g)$.
Thus
$$
\frac{\Delta_{g_+}|F|_{g_+}}{|F|_{g_+}}=-2(d-3)+ \frac{\sigma}{\Phi} \Big(\sigma\Big[\Delta^g-
\frac{4(d-3)}{d}J^g\Big] -(d-6)\Big[\nabla_n +\frac{2}d \nabla.n\Big]
\Big)\Phi\, .
$$
Hence, because the Yang--Mills equation implies $F_{ab}|_{\partial M}=\bar F_{ab}$, the right hand side  extends to a smooth function on $M$, equaling $-2(d-3)$  on~$\partial M$.
\end{proof}

Hence Equation~\nn{FROGGO} now becomes (setting $d=6$)
\begin{equation}\label{IHSE}
(\Delta^{g_+}+6 -V)f^{\prime}=-|F|_{g_+}\, ,
\end{equation}
where $V$ is a short range potential (meaning $V\in C^\infty M$ and $V|_{\partial M}=0$, so that $V$ is bounded by compactness of $M$). Equation~\nn{IHSE} is an inhomogeneous Schr\"odinger equation on a six dimensional Poincar\'e--Einstein manifold.

\section{Scattering Theory}\label{scatter}

The central problem is now to study solutions $f'$  to the inhomogeneous
Schr\"odinger Equation~\nn{IHSE}. For that we need to impose
appropriate boundary conditions.
In particular we would like to achieve that
$$
\tilde g = \tilde \sigma^2 g_+\, ,\:\mbox{ where }
\tilde \sigma = e^f\ \mbox{ and }  f'=|F|_{g_+}f\, ,
$$ is a continuous compactified metric, meaning that
${\partial M}=\tilde\sigma^{-1}(0)$, where $\tilde \sigma\in C^1M$ and is
smooth on the interior $M_+$, and $\ext \tilde \sigma$ is nowhere
vanishing along~${\partial M}$.  In this case~$\tilde g$ is~$C^1$ and its
Levi-Civita connection is continuous. 
Observe that setting
$$
f=(1+K\sigma)\log\sigma + G \sigma\, ,
$$
where $G,K\in C^\infty M$, yields
$$
\tilde g = 
e^{ 2\sigma (G + K \log \sigma)} \sigma^2g_+
=e^{ 2\sigma (G + K \log \sigma)}g \, ,
$$ 
with the required regularity when $\sigma$ is a smooth defining function. This determines the desired boundary regularity and
our problem is now to solve the system
\begin{equation}\label{system}
\left\{
\begin{array}{l}
(\Delta^{g_+}+6 -V)f^{\prime}=-|F|_{g_+}\, ,\\[1mm]
f^\prime=|F|_{g_+}\big((1+K\sigma)\log\sigma + G \sigma\big)\, ,
\end{array}
\right.
\end{equation}
where $G,K\in C^\infty M$. The above boundary behavior will also follow from 
a scattering problem studied below. One may view $K$
as capturing a particular solution to the inhomogeneous equation while
$G$ captures non-local Neumann data for the homogeneous solution.

In fact, the system~\nn{system} has a unique smooth solution for~$f^\prime$.
 This statement follows almost directly from work of Mazzeo and Melrose~\cite{MM}, its generalization to include short range potentials by Joshi and S\'a Barreto~\cite{BaJo}, and the scattering theory of Graham and Zworski~\cite{GrZw}. The reader wishing to treat this as a ``black box'', may skip to Section~\ref{renorm} where we complete the proof of Theorem~\ref{main}.  The remainder of this section sketches the main details required to prove the following.
\begin{theorem}\label{mainth}
Let $(M_+,g_+)$ be a Poincar\'e--Einstein six-manifold and
$A$ a smooth connection on  a vector bundle on $M$, such that
 $|F|^2_{g_+}$ is nowhere vanishing on $M_+$,
 and  $|\bar F|^2_{\bar g}$ is nowhere zero on $\partial M$.
  Given $\sigma$ any smooth defining function for~$\partial M$,  there is a unique solution $f^\prime$ to 
$$
(\Delta^{g_+}+6 - V)f^\prime = -|F|_{g_+}\, ,
$$  
of the form
\begin{equation}\label{asyms}
f^\prime=\sigma^2\big(f_0 \log \sigma + f_1\sigma\big)\, ,
\end{equation}
for $f_0, f_1\in C^\infty M$ and $f_0|_{\partial M} = \bar \Phi$ where $\bar \Phi=|F|_g\big|_{\partial M}$.
\end{theorem}

\begin{proof}
Let us first work in  general dimensions $d\geq 6$, and  consider the ``scattering equation''
\begin{equation}\label{scatteringequation}
\big(\Delta^{g_+,A}+(s-2)(d-3-s)\big)u(s) =0 \, ,
\end{equation}
where the spectral parameter $s\in {\mathbb C}$. Observe, by virtue of Equation~\nn{selfadjoint}, 
that  
$
u=|F|_{g_+}$
solves the above partial differential equation when $s=2,d-3$.
Moreover, our assumption on~$|\bar F|^2_{\bar g}$ implies that
$|F|_{g_+}\notin L^2(M_+,g_+)$.

Next we examine bound states, namely solutions $u\in L^2(M_+,g_+)$ with respect to the volume element $\ext {\rm Vol}(g_+)$. Suppose that $u$ is such a (non-trivial) bound state solution. Then   by multiplying the scattering Equation~\nn{scatteringequation}  with $u^*$, integrating over $M_+$, performing an  integration by parts (using that $\Delta^{g_+}$ was defined in the first place by considering a suitable dense set of Schwartz functions, say) employing Equation~\nn{selfadjoint}, we may reexpress the bound state eigenvalues of $\Delta^{g_+,A}$ as 
\begin{equation}\label{bound}
(s-2)(s-d+3)=\Big(s-\frac{d-1}{2}\Big)^2 - \frac{(d-5)^2}4=-
\frac{\int_{M_+} \ext {\rm Vol}_{g_+}
\Big||F|_{g_+} \nabla \big(\frac{u}{|F|_{g_+}}\big)\Big|_{g_+}^2
}{\int_{M_+} \ext {\rm Vol}{g_+} |u|_{g_+}^2}\, .
\end{equation}
The right hand side is real and strictly negative (recall that $|F|_{g_+}\notin L^2(M_+,g_+)
$) so we learn that 
bound states (with real values of $s$) occur when 
$$
s\in (2,d-3)\, .
$$

\smallskip

As discussed earlier (see Equation~\nn{thecat} and Lemma~\ref{thepot}), the operator $\Delta^{g_+,A}-2(d-3)$ equals the Schr\"odinger operator $\Delta^{g_+} -V$, so that the scattering operator appearing in Equation~\nn{scatteringequation} becomes
\begin{equation*}
\Delta^{g_+} -V + s(d-1-s)\, ,
\end{equation*} whose resolvent---in the zero $V$ case---was constructed by Mazzeo and Melrose in~\cite{MM}. That construction allowed Graham and Zworski to make the corresponding meromorphic Poisson operator ${\mathscr P}(s)$~\cite{GrZw}. Both the resolvent and Poisson operator in the case where the potential $V$ is non-zero  and short range   have been constructed by Joshi and S\'a Barreto~\cite{BaJo}. 
By the spectral theorem, the resolvent is holomorphic for $\operatorname{Re}s>d-3$. 
Mazzeo and Melrose have established that the operator $\Delta^{g_+}$ has only finitely many bound states~\cite{Mazzeo,MM}.
We must now consider the case where this operator is modified by a short range potential
(perturbation theory for Schr\"odinger operators is treated in~\cite{Reed,Kato}).
The finitely many bound states property is preserved upon  addition of the short range potential~$V$; this can be  proved by  
adjusting the parametrix construction leading to the analytic Fredholm-type theorem
of~\cite{MM} to include $V$~\cite{MMC}. See also~\cite{BaJo} where much of this construction was performed.
 Equation~\nn{bound} ensures that the lowest eigenvalue of $-\Delta^{g_+,A}$ is strictly greater than zero.
Hence the perturbed resolvent and associated Poisson operator are holomorphic at $s=d-3$.
Thus there exists an operator ${\mathscr P}(s):C^\infty{\partial M} \to C^\infty M_+$ such that given $\bar u \in C^\infty{\partial M} $,
$$
u(s)={\mathscr P}(s)\bar u
$$
that is holomorphic at $s=d-3$
 such that  
$$
\big(\Delta^{g_+} -V + s(d-1-s)\big) u(s) = 0\, .
$$
Moreover, when $s\notin\{\frac{d-1}{2}, \frac{d-1}2+1, \frac{d-1}2 +2 ,\cdots\}$, 
\begin{equation}\label{us}
u(s) = \sigma^{d-1-s} u_0(s)+ \sigma^s u_1(s) \, ,
\end{equation}
where $u_0,u_1\in C^\infty M$. [Note that the analysis leading to the above behavior is by now quite standard. Either one can employ a normal form for the Poincar\'e--Einstein metric~\cite{GLee}, or use that the scattering operator $\Delta^{g_+} - V +s(d-1-s)$ can be expressed as $\sigma$ times a Laplace--Robin-type operator subject to an $\mathfrak{sl}(2)$ algebra of the kind discussed in~\cite{GW}.]
In particular, $u=|F|_{g_+}$ is the unique solution when $s=d-3$ and such that~$u_0|_{\partial M}$ equals the restriction $\bar \Phi$ of the  smooth extension $\Phi \in C^\infty M$ of $|F|_{\sigma^2 g_+}$ to~${\partial M}$.

Since $u(s) = {\mathscr P}(s)\bar\Phi$ is holomorphic around $s=d-3$, 
 and $u(s)$ obeys the scattering Equation~\nn{scatteringequation},
we now consider $$f'=-\Big(\frac{\ext }{\ext s} {\mathscr P}(s)\bar\Phi\Big) \Big|_{s=d-3}\, .$$ 
Differentiating~\nn{scatteringequation} with respect to $s$ and setting $s=d-3$ then gives
$$
\Delta^{g_+,A} f^\prime +(d-5) u(d-3)=0\, .
$$
As discussed above, $u(d-3)$ must equal $|F|_{g_+}$. Hence it only remains to establish Equation~\nn{asyms} and uniqueness of the solution. The former follows
 from  differentiating Equation~\nn{us} with respect to $s$, setting~$s=d-3$ and remembering that $u_0|_{\partial M}=\bar \Phi$ which is $s$-independent.
 The latter relies on Lemma~\ref{asym}, given  below, which shows that the leading asymptotics of $f^\prime$ are determined to sufficient order in order  that the difference of any pair of solutions is in $L^2(M_+,g_+)$. Since such a difference must obey  
 the $s=d-3$ case of the scattering Equation~\nn{scatteringequation} and bound states are  in the interval $s\in (2,d-3)$, 
uniqueness follows. The proof is completed by setting $d=6$.
  \end{proof}

%
%
%
%
%

%
%
%
%
%
%
%
%
%
%
%

\section{Renormalization}\label{renorm}

We are now ready to tackle the proof of Theorem~\ref{main}. 
We first establish that the Neumann data for the inhomogeneous Laplace-type Equation~\nn{IHSE} encodes the renormalized energy.  
For that we need to analyze the asymptotics of the solution $f^\prime$ to the problem~\nn{system}.
\begin{lemma}\label{asym}
Let $(M_+,g_+)$ be a Poincar\'e--Einstein $6$-manifold, $A$ a smooth vector bundle connection on  a vector bundle 
such that  
  $|F|_{g_+}$ is nowhere zero on $M_+$ and  $|\bar F|^2_{\bar g}$ is nowhere zero on $\partial M$.  Let~$\sigma$ a smooth defining function for~${\partial M}$. 
Then
$$
f^{\rm asym}=|F|_{g_+}\big((1+K\sigma)\log\sigma + G \sigma\big)\, ,
$$
where $G\in C^\infty M$ and 
$$K|_{\partial M}=
-\frac{2  (\nabla_{\hat n} + 2H^g) |F|_g}{|F|_g}\, ,
 $$
obeys
$$
\Delta^{g_+,A}f^{\rm asym}=-|F|_{g_+}+ {\mathcal O}(\sigma^4 \log \sigma,\sigma^4)\, .
$$
\end{lemma}
\begin{proof}
Firstly one finds
\begin{align*}
\Delta^{g_+,A} f^{\rm asym}&
=|F|_{g_+}  \hh g^{ab}_+ \big(\nabla_a
+2v_a\big) \Big(\nabla_b \big((1+K\sigma)\log\sigma + G \sigma\big)\Big)\\
&
=|F|_{g_+}  \big( \sigma^2 \Delta  + 2 \sigma^2 \nabla_v\big)  \big((1+K\sigma)\log\sigma + G \sigma\big)\, ,
\end{align*}
where $v:=\ext \log |F|_g$, $n:=\ext  \sigma$ and, 
unless specified otherwise,  all index manipulations and Levi-Civita connections on the right hand side are with respect to the compactified metric $g=\sigma^2 g_+$.

Now observe that 
\begin{align*}
\sigma^2 \Delta \sigma\;\; &= \sigma^2 \nabla.n=-6\sigma^2\rho -\sigma^3 J \, ,
\\
\sigma^2 \Delta \log\sigma\:  &= 
-|n|^2 + \sigma \nabla.n
=-1 -4 \sigma\rho
-\sigma^2 J
\, ,\\ 
\sigma^2 \Delta (\sigma \log\sigma) &= 
\sigma |n|^2 
+ (\sigma^2  \log \sigma 
+\sigma^2) \nabla.n=
\sigma -2 \sigma^2 ( 3 \log \sigma +4) \big(\rho+{\mathcal O}(\sigma)\big)
\, ,
\end{align*}
where $-6\rho :=\nabla.n + \sigma J$. In the above, we have used Equation~\nn{wehavegoals}.
Hence 
\begin{align*}
\Delta^{g_+,A} f^{\rm asym}&=
|F|_{g_+}\Big(
-1 + (K-4\rho + 2n.v ) \sigma
+{\mathcal O}(\sigma^2 \log\sigma,\sigma^2) \Big)\, .
\end{align*}
Noting that $|F|_{g_+}={\mathcal O}(\sigma^2)$ and that the Poincar\'e--Einstein condition implies  $\rho|_{\partial M} = -H^g$ (see for example~\cite{Goal}), the result follows.
 \end{proof}

We now need a technical lemma.
\begin{lemma}\label{thebluelemma}
Let $(M_+,g_+)$ be a Poincar\'e--Einstein $6$-manifold, $A$ a smooth vector bundle connection 
that solves the Yang--Mills equation $
j[A,g_+]=0
$
on the interior~$M_+$, and $\sigma$ a smooth defining function for~${\partial M}$. Then
$$
\Big((\nabla_{\hat n}+4H^g)|F|^2_g\Big)\Big|_{\partial M}=0\, .
$$
In the above $\hat n$ is the unit normal to the hypersurface ${\partial M}$ with respect to $g=\sigma^2 g_+$ and~$H^g$ the corresponding mean curvature. It follows that
the function $K$  given in Lemma~\ref{asym}
obeys
$$
K|_{\partial M} =0\, . 
$$
\end{lemma}
\begin{proof}
Let $n:=\ext \sigma$ and consider
\begin{align*}
\frac12 \nabla_n |F|^2_g &= 
-\operatorname{Tr} F^{ab} \nabla_n F_{ab}=
 2\operatorname{Tr} F^{ab} n^c\nabla_a F_{bc}\\
 &=2\operatorname{Tr} \big(F^{ab} \nabla_a(n^c F_{bc})
 -F^{ab} (\nabla_a n^c) F_{bc}\big
 )\, .
\end{align*}
But along ${\partial M}$, on a Poincar\'e--Einstein manifold we have (see for example~\cite{Will1})
$$
\nabla_a n_b |_{\partial M}=
H^g\hh g_{ab}\, ,
$$
while the Yang--Mills condition implies $n^c F_{bc}=-\frac12 \sigma \nabla^c F_{cb}$ (see Equation~\nn{ccYM} below). 
Hence
\begin{align*}
\frac12 \nabla_{\hat n} |F|^2_g\big|_{\partial M}=
\operatorname{Tr} \big( -F^{ab} \hat n_a \nabla^c F_{cb}
-2H^g F^{ab} F_{ba}\big)\big|_{\partial M}=-2 H^g |F|_g^2\big|_{\partial M}\, .
\end{align*}
The first result now follows while the second needs only the product rule.
\end{proof}

%
%
 
 Using that $f^\prime$ solving the system~\nn{system}, for a given smooth defining function $\sigma$, is uniquely determined by Theorem~\ref{mainth}, it follows that the function 
 $$\overline G:=G|_{\partial M} \in C^\infty$$ is also determined.  In fact it encodes the renormalized energy.
 
 \begin{theorem}\label{BoundaryB}
 Let $(M^6_+,g_+)$ be a  Poincar\'e--Einstein manifold  and $A$ be a smooth vector bundle connection on $M$. Moreover suppose that on $M_+$  that the Yang--Mills equation $j[A,g_+]=0$ holds, $|F|^2_{g_+}$ is nowhere vanishing, and 
    $|\bar F|^2_{\bar g}$ is nowhere zero on $\partial M$.
 Then
$$
\frac14\color{black}
\int_{\partial M} 
 \ext {\rm Vol}(\bar g) 
|F|_{{g^{\rm c}}}^2\hh  \overline G = 
E_{\rm ren}\, .
$$
 \end{theorem}

\begin{proof}
We begin by integrating  $-|F|_{g_+}$ times the   inhomogeneous Schr\"odinger Equation~\nn{FROGGO} over the cutoff interior $M_{\varepsilon}$
$$
-\int_{M_{\varepsilon}}   \ext {\rm Vol}(g_+ )  |F|_{g_+} \Delta^{A,g_+} f^\prime = \int_{M_{\varepsilon}}  \ext {\rm Vol}(g_+ ) |F|_{g_+}^2\, .
$$
Applying the renormalizing operator 
$\operatorname{ren}:= \frac{\ext}{\ext\varepsilon}\circ \hh\varepsilon\: \:\pdot\:\: \Big|_{\varepsilon=0}$ to both sides and using the divergence theorem on the left hand side gives
$$
\operatorname{ren}\int_{{\partial M}_{\varepsilon}}
 \ext {\rm Vol}(\bar g_+ )
 \hat n^a_{g_+,\varepsilon} |F|_{g_+}^2 
\nabla_a \Big( \frac{f^\prime}{\:|F|_{g_+}}\Big)
  = 
   4\color{black}
E_{\rm ren}\, .
$$
Here we used that
$|F|_{g_+}\Delta^{g_+,A} :=\nabla_a^{g_+} \circ g_+^{ab} |F|_{g_+}^2 
\nabla_b^{g_+} \circ \frac{1}{\:|F|_{g_+}}$ from Equation~\nn{selfadjoint} and $\hat n^a_{g_+,\varepsilon}$ is the inward unit normal vector.
We now use that the solution $f'$ to the inhomogeneous 
Schr\"odinger equation can be written in the form
$$
f^\prime=|F|_{g_+}\big((1+K\sigma)\log\sigma + G \sigma\big)\, ,
$$
so that
\begin{align*}
E_{\rm ren}&=
\frac14\color{black}
\operatorname{ren}\int_{{\partial M}_{\varepsilon}}
 \ext {\rm Vol}(\bar g_\varepsilon )\hh 
 |F|_{{g^{\rm c}}}^2 \hh\hh
\nabla_{\hat n_\varepsilon} 
\big((1+K\sigma)\log\sigma + G \sigma\big)\, .
\end{align*}
Here we have written all quantities in terms of the canonical compactified metric ${g^{\rm c}}=\sigma^2_{\rm c} g_+$ determined by $\bar g$.
For this,  recall that $\sigma|_{{\partial M}_{\varepsilon}}=\varepsilon$ so $g_+$ pulled back to ${\partial M}_\varepsilon$ gives $\varepsilon^{-2} \bar g_\varepsilon$ where $\bar g_\varepsilon$ is the pullback of the compactified metric.

Next we use that  $K|_{\partial M} = 0$ (see Lemma~\ref{thebluelemma}) so that 
$K|_{{\partial M}_{\varepsilon}} = {\mathcal O}(\varepsilon)$.  
Also $
\ext \sigma |_{{\partial M}_{\varepsilon}}
$
is  a conormal to~${\partial M}_\varepsilon$ and along ${\partial M}$  has unit length. Hence 
$$\nabla_{\hat n_\varepsilon}\sigma |_{{\partial M}_\varepsilon}=1+\mathcal{O}(\varepsilon)\, .$$
It is now easy to check that
$$
E_{\rm ren}=
\frac14\color{black}
\int_{\partial M}  \ext {\rm Vol}(\bar g )\hh 
 |F|_{{g^{\rm c}}}^2 G\, .
$$
\end{proof}

We need one last technical lemma.

\begin{lemma}\label{bdy}
Let $(M^d_+,g_+)$ be a  conformally compact manifold with $d\neq 4$,  and $A$ be a smooth connection on a vector bundle ${\mathcal V}M$ that solves the Yang--Mills equation,
$$
j[A,g_+]=0\, ,
$$
on the interior $M_+$. Then for any 
smooth defining function $\sigma$ for $\partial M$ and corresponding
cut-off interior $M_{\varepsilon}$,
$$
\int_{M_\varepsilon} 
 \ext {\rm Vol}(g_+ ) g_+^{ab}\nabla_a^{g_+}
 \operatorname{Tr} ( F_{bc}  \mu^c)
\stackrel{\varepsilon\to 0}\longrightarrow 0\, ,
$$
when $\mu\in \Gamma(TM\otimes \End{\mathcal V}{M})$ obeys
$$
\sigma^{3-d} \mu|_{\partial M_{\varepsilon}} \stackrel{\varepsilon \to 0}{\longrightarrow}0\, .
$$
\end{lemma}

\begin{proof}
On $M_+$, the Yang--Mills equation can be rewritten in terms of a compactified metric $g=\sigma^2 g_+$, as~\cite{GLWZ}
\begin{equation}\label{ccYM}
j[A,g_+]_a:=\sigma\big(\sigma g^{bc}\nabla_b F_{ca}-(d-4)n^b F_{ba}\big)=0\, ,
\end{equation}
where $n:=\ext \sigma$, $n^a = g^{ab} n_b$, and the compactified metric $g = \sigma^2 g_+$ is everywhere used on the right hand side. Because $\sigma=\varepsilon$ along ${\partial M}_\varepsilon$ and $\sigma$ is defining for ${\partial M}= \partial M_0$, for any small enough~$\varepsilon$ there exists $0<k\in C^\infty M$ such that  the unit inward pointing normal to ${\partial M}_\varepsilon$ with respect to $g$ obeys $\hat n^a_\varepsilon=kn^a$. Hence, from the Yang--Mills equation, along~${\partial M}_\varepsilon$  we have
$$
\hat n_{\varepsilon}^b F_{ba} = \tfrac{k}{d-4}\varepsilon g^{bc} \nabla_b F_{ca}\, .
$$
So by the divergence theorem, 
$$
\int_{M_\varepsilon} 
 \ext {\rm Vol}(g_+ )\hh g_+^{ab}\nabla_a^{g_+}
 \operatorname{Tr} ( F_{bc}  \mu^c)
=-
\frac{\varepsilon^{3-d}}{d-4}
\int_{{\partial M}_\varepsilon} \ext \!\operatorname{Vol}(\bar g_\varepsilon) k 
\operatorname{Tr} (\mu^ag^{bc} \nabla_b F_{ca})\, ,  
$$
where $\bar g_\varepsilon$ is the metric induced along ${\partial M}_{\varepsilon}$ by $g$,  and we have used that the unit normal~$\hat n^{g_+}_\varepsilon$ with respect to $g_+$ obeys $\hat n^{g_+}_\varepsilon= \varepsilon \hat n^{g}_\varepsilon$. The result now follows.
\end{proof}

We are  ready to prove our main Theorem~\ref{main}.
\begin{proof}
We begin with the compactified metric
\begin{equation}\label{metrics}
\tilde g =e^{2f}g_+= e^{\frac{2f^\prime}{|F|_{g_+}}}g_+
=e^{2(1+K\sigma)\log\sigma + 2G \sigma}g_+
=e^{ 2\sigma (G + K^\prime\sigma \log \sigma)} g\, ,
\end{equation}
where we have written $K=K^\prime \sigma$ with $K^\prime\in C^\infty M$ because $K|_{\partial M}=0$ (see Lemma~\ref{thebluelemma}).
As $f^\prime$ solves the the system in Equation~\nn{system}, we have by Lemma~\ref{shift} that
$$
e^{6f} Q[{\tilde g},A] = -2 \operatorname{Tr}\big(\nabla^c_{g_+}(F^{ab}_{g_+} F_{ca}  \nabla_{b}f)\big)\, .
$$
Integrating the above over the cutoff interior with respect to the metric $g_+$ gives
$$
\int_{M_\varepsilon }  \ext {\rm Vol}(\tilde g ) Q[{\tilde g},A] =
-2\int_{M_\varepsilon }  \ext {\rm Vol}( g_+ ) g_+^{ab}
\nabla_a^{g_+}\operatorname{Tr}\big(
F_{bc}
F^{cd}_{g_+}  \nabla_{d}f\big)\, .
$$
Because 
$$\sigma^{-3} F_{g_+}^{cd} \nabla_d f 
=\sigma F_g^{cd} \nabla_d\big((1+\sigma^2 K^\prime)\log\sigma
+G \sigma\big)\, ,
$$
it follows that $(\sigma^{-3} F_{g_+}^{cd} \nabla_d f )|_
{{\partial M}_{\varepsilon}}$
is order ${\mathcal O}(\varepsilon)$ (remember that the Yang--Mills equation implies $2n^a F_{ab} = \sigma \nabla^a F_{ab}$).
So by Lemma~\ref{bdy} (with $d=6$) 
 it follows that $$\lim_{\varepsilon\to 0}\int_{M_\varepsilon }  \ext {\rm Vol}(\tilde g )\hh  Q[{\tilde g},A] =0\, .$$

On the other hand by Lemmas~\ref{E2Q}
and~\ref{Q2ren}
we have that
$$
Q[{\tilde g},A] = {\mathcal E}_{\rm ren}^{\tilde g}
+\operatorname{Tr}\Big(\frac14 \Delta^{\tilde g} (F^{ab}F_{ab})_{\tilde g}
+\nabla_a^{\tilde g} (F^{ab} j_b)_{\tilde g}
\Big)\, ,$$
so we now analyze
$$
\int_{M_\varepsilon }  \ext {\rm Vol}(\tilde g )
\operatorname{Tr}\big(\nabla_a^{\tilde g} (F^{ab} j_b)_{\tilde g}\big)
=-
\int_{{\partial M}_\varepsilon }  \ext {\rm Vol}(\bar{\tilde g}_\varepsilon )\hat n^a_{\tilde g,\varepsilon}
\operatorname{Tr} (F_{ab} j^b)_{\tilde g}
\, .
$$
On the right hand side we employed the divergence theorem. Both the metric $\tilde g$ and its Levi-Civita connection are continuous on $M$. The Yang--Mills equation can then be used to  show 
that $\hat n^a_{\tilde g,\varepsilon}
F_{ab}$ is ${\mathcal O}(\varepsilon)$ so
 the last display vanishes as $\varepsilon\to 0$.
 
 \smallskip
 
Next we analyze 
$$
\frac14 \int_{M_\varepsilon }  \ext {\rm Vol}(\tilde g ) \operatorname{Tr} \Delta^{\tilde g} (F^{ab}F_{ab})_{\tilde g}
=\frac14 \int_{{\partial M}_{\varepsilon}}
 \ext {\rm Vol}(\bar{\tilde g}_{\varepsilon} ) 
  \nabla_{\hat n_{\tilde g, \varepsilon} } |F|^2_{\tilde g}\:
  \stackrel{\varepsilon\to 0}
  {\longrightarrow}\:
  \frac14 \int_{{\partial M}}
 \ext {\rm Vol}(\bar{ \tilde g} ) 
  \nabla_{\hat n_{\tilde g} } |F|^2_{\tilde g}
  \, .
$$
The above limit again relied on continuity of $\tilde g$ and 
its Levi-Civita connection on $M$, and~$\bar {\tilde g}$ is the pullback of $\tilde g$ to $\partial M$.
We can now utilize Lemma~\ref{thebluelemma} to write
the above right hand side as
$$
-\int_{{\partial M}}
 \ext {\rm Vol}(\bar{ g}) 
  H^{\tilde g} |F|^2_{g}\, .
$$
Note that $\tilde g$ agrees with $g$ along ${\partial M}$, so we have replaced $\tilde g$ everywhere by $g$ save in the mean curvature. Now, from Equation~\nn{metrics}, 
$\tilde g =\tilde \Omega^2 g$ where  
$$
\tilde \Omega = e^{ \sigma (G + K^\prime\sigma \log \sigma)}\, . 
$$
But 
$$
H^{\tilde\Omega^2 g}= H^g + \tilde \Upsilon. \hat n_g\big|_{\partial M}\, , 
$$
where $ \tilde \Upsilon=\ext \log \tilde\Omega$; here we used that $\tilde \Omega|_{\partial M}=1$. 
In particular, we now choose for $g$ the canonical compactified metric ${g^{\rm c}}$ which happens to obey $H^{g^{\rm c}}=0$. Also
$$ \tilde \Upsilon. \hat n_g\big|_{\partial M}=G|_{\partial M}\, ,$$
because $\nabla_{\hat n} \sigma|_{\partial M}=1$.
Hence, by Theorem~\ref{BoundaryB},
$$
\frac14 \int_{M_\varepsilon }  \ext {\rm Vol}(\tilde g ) \operatorname{Tr} \Delta^{\tilde g} (F^{ab}F_{ab})_{\tilde g}
\stackrel{\varepsilon\to 0}
  {\longrightarrow}\:
  -
  \int_{\partial M}  \ext {\rm Vol}(\bar g )\hh 
 |F|_{{g^{\rm c}}}^2 \hh\hh \overline G =  -4E_{\rm ren}\, .
  $$

The last term to analyze is $\int_{M_\varepsilon }  \ext {\rm Vol}(\tilde g )
 {\mathcal E}_{\rm ren}^{\tilde g}$. The integrand here is a pointwise conformal invariant, so we may employ {\it any} metric $g$ in  the interior conformal class to evaluate it, and then take the $\varepsilon\to 0 $ limit. Thus
 $$
 E_{\rm ren} = \frac14\int_{M }  \ext {\rm Vol}( g )
 {\mathcal E}_{\rm ren}^{g}\, .
 $$

\end{proof}

\section{Examples}\label{exes}

As pointed out earlier, the power of our theory is that it applies to any vector bundle connection on a Poincar\'e--Einstein manifold.
We consider two examples. The first is for the tractor bundle equipped with its canonical tractor connection; see~\cite{BEG}. This realizes an interesting link with geometry intrinsic to the Poincar\'e--Einstein structure. 
The second example  of Maxwell's equations on the hyperbolic ball is a model for the general theory that is simple enough that relevant quantities can be computed explicitly, and so provides a nice check of our main results.

\subsection{The Tractor Connection}

A conformal $d$-manifold $(M,[g])$ canonically determines both a rank $d+2$ vector bundle $\cT M$, termed the tractor bundle~\cite{BEG}, along with 
a connection $A([g])$ termed the tractor connection~\cite{Thomas}. 
For a detailed discussion of tractor bundles and their associated calculus, see~\cite{BEG,CG}. Here we only need a few key ingredients. Firstly, given a choice of metric $g\in [g]$, a section $V^A\in \Gamma(\cT M)$ is determined by a triple
$$
V^A=\begin{pmatrix} v^+\\ v^a \\ v^-\end{pmatrix} \, ,
$$
where $v^\pm \in C^\infty M$ and $v^a\in \Gamma(TM)$.
The one-form valued tractor $\nabla_a^{A([g])} V^B\in \Gamma(T^*M \otimes \cT M)$ is then given by the triple
\begin{equation}\label{Tconn}
\nabla_a V^A = \begin{pmatrix} \nabla_a v^+
- v_a \\ 
\nabla_a v^b + P{}^b_a v^+ + \delta_a^b v^-  \\ \nabla_a v^--P{}_{ab}v^b\end{pmatrix}\, , 
\end{equation} 
where all Riemmannian quantities on the right hand side are computed with respect to the metric $g$. The curvature of the tractor connection is then given, in an obvious matrix notation,  by the endomorphism-valued two form 
$$
F_{ab}{}^C{}_D=
 \begin{pmatrix}
0&0&0\\
C_{ab}{}^c&W_{ab}{}^c{}_d&0\\
0&-C_{abd} &0
\end{pmatrix}\, .
$$
Now, if $h$ is some Riemannian metric (not necessarily an element of $[g]$), the current
\begin{multline*}
\!\!\!j\big[A([g]),h\big]{}_b{}^C{}_D=\quad\quad
\\[1mm]
\:\:\quad
 \begin{pmatrix}
\mbox{ \tiny
$-h^{ae}C_{abe}$}&
 \mbox{ \tiny
$-h^{ae}W_{abed}$}&
  \mbox{ \tiny
$0$}\\
\! \!\!\mbox{ \tiny
$
h^{ae}(\nabla_e C_{ab}{}^c\!\!-\!2K_{e[a}^f C_{|f|b]}{}^c
\!\!-\!P_e^fW_{ab}{}^c{}_f)$}\!\!&
\mbox{ \tiny
$
h^{ae}(
\nabla_eW_{ab}{}^c{}_d\!
-\!2K_{e[a}^f W_{|f|b]}{}^c{}_d\!
-\!\delta_e^c C_{abd}\!+\!C_{ab}{}^c g_{ed}) $}\!&
\mbox{ \tiny
$\!-\hh\!h^{ae}W_{ab}{}^c{}_e$} \!\!\\
\mbox{ \tiny
$0$}&
\mbox{ \tiny
$-h^{ae}(\nabla_e C_{abd}
+2K_{e[a}^f C_{|f|b]d}
-P_{ef}W_{ab}{}^f{}_d )$}&
\mbox{ \tiny
$ h^{ae}C_{abe}$}
\end{pmatrix}\! .\!\!\!
\end{multline*}
In the above, $h^{ab}$ denotes the inverse of $h_{ab }$ while all other quantities are computed with respect to the metric $g$. In addition, the tensor $K=\nabla^h-\nabla^g$ is the contorsion of the Levi-Civita connection of $h$ relative to that of $g$.

Specializing to the case $h=g_+=\sigma^{-2} g$ is a Poincar\'e--Einstein metric on $M_+$, then 
\begin{equation}\label{Jtract}
\sigma^{-1}\!
j\big[A([g]),g_+\big]_b{}^C{}_D=
\begin{pmatrix}
\scriptstyle0&\scriptstyle0&\scriptstyle0\\\!
\scriptstyle\sigma\nabla^a C_{ab}{}^c
-(d-4) C_{nb}{}^c
-\sigma P^{af}W_{ab}{}^c{}_f
\!\!&
\scriptstyle(d-4)(\sigma  C^c{}_{db}
-W_{nb}{}^c{}_d)
 &
 \scriptstyle0\\
\scriptstyle0&
\scriptstyle\!\!\!-\sigma\nabla^a C_{abd}
+(d-4) C_{nbd}
+\sigma P^{af}W_{abdf}\!
&
\scriptstyle0
\end{pmatrix}\! .
\end{equation}
Notice that the expression on the right hand side above 
extends smoothly to $M$.
Also, in the above we used that
both the Weyl and Cotton tensors are trace-free and that
the contorsion of $g_+$ relative to $g$ is given by
$
K_{eb}^a = -\sigma^{-1}(n_e \delta^a_b + n_b\delta^a_e - g_{eb} n^a)\, ,
$
where $n:=\ext \sigma$. In fact, the right hand side above vanishes identically. When $d=4$, this claim follows because $\nabla^a C_{abc} - P^{af} W_{abcf}$ equals the Bach tensor, which vanishes for any Einstein metric. The $d \neq 4$ case of the claim is proved in~\cite{Goal}.
Hence, we have just established the following result. 
\begin{proposition}
Let $(M_+,g_+)$ be a Poincar\'e--Einstein manifold where $g=\sigma^2 g_+$ with~$\sigma$ a defining function for $\partial M$. Then, the tractor connection $A([g])$ obeys the Yang--Mills equation
$$
j\big[A([g]),g_+\big]=0\, .
$$
\end{proposition}

A short computation shows that for the tractor connection, 
$$
|F|_{g_+}^2 = W^{g_+}_{abcd} W_{g_+}^{abcd}\, ,
$$
so the renormalized Yang--Mills energy computes the 
renormalization of the integrated square of the Weyl tensor,
$
\operatorname{ren}\frac14
\int_{M_\varepsilon}\ext {\rm Vol}(g_+)
W^{g_+}_{abcd} W_{g_+}^{abcd}
$,
where as usual, the family of cut-off interiors $M_\varepsilon$ is determined by the canonical compactified metric~${g^{\rm c}}$.
Hence we have the following corollary of Theorem~\ref{main}.
\begin{corollary}
Let $(M^6_+,g_+)$ be a Poincar\'e--Einstein manifold where $g=\sigma^2 g_+$ for $\sigma$ a defining function for $\partial M$. Then, 
\begin{multline}\label{WW}
\operatorname{ren}
\frac14\int_{M_\varepsilon}\ext {\rm Vol}(g_+)\,
W^{g_+}_{abcd} W_{g_+}^{abcd}
\\
=
\int_M  {\rm Vol}(g)\Big(-\frac12
(\nabla^e W^{abcd})\nabla_e W_{abcd}
+8W^{abcd}\nabla_{d}C_{abc} 
\\
-8 C_{abc} C^{abc}
-8 W^{abce}P_{ef}W_{abc}{}^f
\\
+\frac12W^{abef} W_{abcd} W^{cd}{}_{ef}
+2W^{abcd}W_{aecf} W_{b}{}^e{}_d{}^f
  \big)\, .
\end{multline}
\end{corollary}

\begin{proof}
We need to evaluate the six terms on the right hand side of Equation~\nn{Eren}
when~$A$ is the tractor connection. Let us work in $d$ dimensions, and set $d$ to six at the end of our computation. For that notice that the current $j_b{}^C{}_D(g)$ can be computed by  setting $\sigma =1$ in Equation~\nn{Jtract}, which gives 
$$
j_b{}^C{}_D(g)=
 \begin{pmatrix}
0&0&0\\
\nabla^a C_{ab}{}^c-P^{ef}W_{eb}{}^c{}_f&(d-4)C^c{}_{db}&0\\
0&-\nabla^a C_{abd}+P^{ef}W_{ebdf} &0
\end{pmatrix}
$$
and in turn 
$$\operatorname{Tr} j^a j_a
=
(d-4)^2 C_{cda}C^{dca}=-(d-4)^2 C_{abc}C^{abc}\, .
$$
Next, it is not difficult to employ Equation~\nn{Tconn} to compute $\nabla_a j_b$, from which we  find
$$\operatorname{Tr} F^{ab}\nabla_a j_b
=-(d-4)\big(
 C^{abc}C_{abc}
-W^{abcd}\nabla_{d}C_{abc} 
\big)\, .$$
Computing the traced square of the tractor gradient of tractor curvature is more involved. For that we first observe that
$$
\nabla^{h,A}_e\! F_{ab}{}^C{}_D
\stackrel g=\!
 \begin{pmatrix}
 \scriptstyle
-C_{abe}&
 \scriptstyle-W_{abed}&
  \scriptstyle0\\
  \scriptstyle\!
\!\nabla_e C_{ab}{}^c-2K_{e[a}^f C_{|f|b]}{}^c-P_e^fW_{ab}{}^c{}_f&\!
\scriptstyle\nabla_eW_{ab}{}^c{}_d
-2K_{e[a}^f W_{|f|b]}{}^c{}_d
-\delta_e^c C_{abd}+C_{ab}{}^c g_{ed} &
\scriptstyle-W_{ab}{}^c{}_e\! \\
\scriptstyle0&
\scriptstyle-\nabla_e C_{abd}
+2K_{e[a}^f C_{|f|b]d}
-P_{ef}W_{ab}{}^f{}_d &
\scriptstyle C_{abe}
\end{pmatrix} \! .
$$
Specializing to $h=g_+$, squaring, and tracing the above leads to  
\begin{align*}
\operatorname{Tr} (\nabla^c F^{ab})\nabla_c F_{ab}&=
2C^{abc}C_{abc}
-4W^{abec}(
\nabla_e C_{abc}-P_e{}^fW_{abcf})
\\ 
&\phantom{=}
+(\nabla^eW^{ab}{}^d{}_c
-g^{ed} C^{ab}{}_c+C^{abd} \delta^e_{c})
(\nabla_eW_{ab}{}^c{}_d
-\delta_e^{c} C_{abd}+C_{ab}{}^c g_{ed})
\\
&
=
-(\nabla^e W^{abcd})\nabla_e W_{abcd}
+2(d-4)C^{abc}C_{abc}+4W^{abcd}\nabla_d C_{abc}
\\ 
&\phantom{=}
-4W^{abce} P_{ef} W_{abc}{}^f\, .
\end{align*}
The remaining cubic terms can be computed by inspection, leading to the quoted result.
\end{proof}

\begin{remark}\label{case}
The renormalization of the integral of the squared Weyl tensor is a special case of the renormalized curvature integrals studied by Case {\it et al} in~\cite{Caseetal}.
Indeed, that work establishes that in six dimensions this quantity is a non-zero multiple of the invariant given in Equation~1.1c of~\cite{Casesix}. It is not difficult to verify that their result matches Equation~\nn{WW} above, upon employing the last identity quoted in the proof of Proposition~2.1 of~\cite{Casesix}.
\end{remark}

\subsection{Maxwell and the Hyperbolic Ball}

A model problem is the case of Maxwell's equations on the hyperbolic ball. 
Our plan is to first construct a global solution to Maxwell's equations,  then to explicitly compute its renormalized energy, and finally compare that result with Theorem~\ref{main}.

Let $M=\overline{B^d}\subset {\mathbb R}^d$ be the closed unit ball (with respect to the Euclidean metric $\delta$) equipped with the hyperbolic metric 
$$
g_+ = \frac{\delta}{\sigma^2}\, ,\:\mbox{ where }
\sigma = \frac{1-r^2}{2}\, ,
$$
on its interior $B^d$.  Here $r$ is the radial Euclidean  distance from the origin. It is easy to check that $g_+$ is a Poincar\'e--Einstein metric. The Maxwell system is the problem of finding a smooth  two-form $F\in \Omega^2 M$ that is closed on $M$ and coclosed with respect to the metric $g_+$ on $M_+$, so that
$$
\nabla_{[a} F_{bc]} = 0 = g_+^{ab} \nabla_a^{g_+} F_{bc}\, .
$$
Note that we also take $F$ to be pure imaginary because it is the curvature of a $U(1)$ connection.

We would like to perform all computations with respect to the Euclidean metric $\delta$, so rewrite the coclosed condition (using Equation~\nn{ccYM}) as
\begin{equation}\label{ccYM1}
0 = \sigma \nabla^a F_{ab} - (d-4) n^a F_{ab}=:\jmath_b\, ,
\end{equation}
where $n=\ext \sigma$. Notice that
$$\delta_{ab}=-\nabla^\delta_a n_b\, .$$ All index manipulations and affine connections are henceforth those of $\delta$.

\medskip

To construct a model solution 
we make an ansatz $i F=\ext B$ where
$$
B_a=-\frac12 f(r^2) \varphi_{ba}n^b\, ,
$$
and the (real) two-form $\varphi$ is parallel with respect to the Levi-Civita connection $\nabla^\delta$ of the compactified metric $\delta$.  
It follows that 
$$
iF_{ab}=f \varphi_{ab} + f' (n_{a} 
\varphi_{nb}-n_b \varphi_{na})\, .
$$
Thus, working from now on in $d=6$ dimensions,
$$
i\hh n^a F_{ab} = (|n|^2 f'+f )\varphi_{nb}
\mbox{ and }
i\hh \nabla^a F_{ab} = -
2(|n|^2 f''+4f') \varphi_{nb}=:i\cdot j_b
\, ,$$
so that Equation~\nn{ccYM1} gives
$$
i\cdot \jmath_b =-\big[(1-|n|^2)(|n|^2 f''+4f')
+2 (|n|^2 f'+f )
\big] \varphi_{nb}\, .
$$

Let us call $z:=|n|^2=r^2$, so we are tasked with solving the hypergeometric equation
$$
(1-z)z f''(z) + (4-2z) f'(z) + 2f(z)=0\, .
$$
Fortunately, there are two rather uncomplicated solutions, $2-z$ and $(1-2z)/z^3$. We focus on the  first as it gives a global solution to our problem,
\begin{equation}\label{Maxsol}
i F_{ab} = (1+2\sigma) \varphi_{ab} - n_a \varphi_{nb}+ n_b \varphi_{na}\, .
\end{equation}

\medskip
We are now ready to compute the renormalized energy explicitly, first observe that
\begin{equation}\label{FF}
|F|_{g_+}^2=\sigma^4\big[(1+2\sigma)^2|\varphi|^2_\delta
-2(1+6\sigma) \varphi^{na}\varphi_{na}\big]\, .
\end{equation}
To construct the family of cut-off interiors corresponding to the canonical compactified metric ${g^{\rm c}}$, 
 we introduce the function $$\rho:M\to [0,1]\quad\mbox{  defined by }\:
\rho:=\frac{1-r}{1+r}\, ,
$$
which satisfies
$$
\frac1\sigma
\frac{\ext r}{\ext \rho }=-\frac{1}\rho\, .
$$
Note that $-\log \rho$ measures the hyperbolic distance to the origin.
The metric $g_+$ then takes the normal form
$$
g_+=\frac{\ext \rho^2 + h(\rho) ds^2_{S^5} }{\rho^2}=\frac{{g^{\rm c}}}{\rho^2}\, ,
$$
where $ ds^2_{S^5}$ is the standard unit round sphere metric and $h(\rho)=(\rho r/\sigma)^2$.
Also note that
$
r(\rho)=(1-\rho)/(1+\rho)
$.
We need to study the $\varepsilon$-family of balls for which
 $\rho\in [\varepsilon,1]$. Hence we let $B_\varepsilon$ be a six-ball of Euclidean radius $r=(1-\varepsilon)/(1+\varepsilon)$.

With  the canonical  family of cut-off interiors in hand, we now consider the regulated energy of~\nn{Ereg}, which  by virtue of Equation~\nn{FF}, is given as
$$
E_{\varepsilon} = \frac14 \int_{B_\varepsilon} 
 \ext {\rm Vol}(\delta)
\frac{(1+2\sigma)^2|\varphi|^2_\delta
-2(1+6\sigma) \varphi^{na}\varphi_{na}}{\sigma^2}\, .
$$ 
Here we used that $ \ext {\rm Vol}(g_+)=
 \ext {\rm Vol}(\delta)/\sigma^6$.
 Now note that $\varphi^{\hat n a} \varphi_{\hat n a}$, where $\hat n:=n/|n|_\delta$, defines a map $S^5\to {\mathbb R}_{\geq 0}$,
and that $|\varphi|_\delta^2$ is a constant.
Thus we have
\begin{align*}
E_{\varepsilon} &= \frac{|\varphi|_\delta^2}4 \int_{B_\varepsilon} 
 \ext {\rm Vol}(\delta)
\frac{(1+2\sigma)^2}{\sigma^2}
-\frac12
 \int_{B_\varepsilon} 
 \ext {\rm Vol}(\delta)
\frac{(1-2\sigma)(1+6\sigma)}{\sigma^2}
 \varphi^{\hat na}\varphi_{\hat na}\\
 &=
 \frac{(3+\varepsilon)(1-\varepsilon)^6(1+3\varepsilon)}{24(1+\varepsilon)^6\varepsilon}
 \, 
 {\rm Vol}(S^5)|\varphi|_\delta^2
 -
 \frac{(1-\varepsilon)^8}{4(1+\varepsilon)^6\varepsilon}
 \int_{S^5} 
 \ext {\rm Vol}(S^5) \varphi^{\hat na}\varphi_{\hat na}\\[1mm]
 &=
 \Big[\frac1{8\varepsilon}-\frac{13}{12}\Big]{\rm Vol}(S^5)|\varphi|_\delta^2
 +
 \Big[-\frac1{4\varepsilon}+\frac{7}2\Big]
  \int_{S^5} 
 \ext {\rm Vol}(S^5) \varphi^{\hat na}\varphi_{\hat na}
 +{\mathcal O}(\varepsilon)
 \, .
\end{align*}
In the above, the volume of a five sphere is ${\rm Vol}(S^5)={\pi^3}$, and $\ext {\rm Vol}(S^5)$ denotes the standard five sphere volume element.  Importantly, as predicted by general theory, there is no $\log\varepsilon$ term in the above.
It now follows that the renormalized energy for the Maxwell solution of Equation~\nn{Maxsol} is 
$$
E_{\rm ren} = 
-\frac{13}{12}\, {\rm Vol}(S^5)|\varphi|_\delta^2
 +\frac{7}2\, 
  \int_{S^5} 
 \ext {\rm Vol}(S^5) \varphi^{\hat na}\varphi_{\hat na}
 \, .
$$
It is not difficult to prove that
\begin{equation}\label{Ipleadthesixth}
\int_{S^5} \ext {\rm Vol}(S^5) \varphi^{\hat na}\varphi_{\hat na}
=\frac16\hh  {\rm Vol}(S^5)|\varphi|_\delta^2\, .\end{equation}
Hence
$$
E_{\rm ren}=-\frac1{2} {\rm Vol}(S^5)|\varphi|_\delta^2\, .
$$

\medskip


Instead of computing {\it na\"ively}, we can also apply Theorem~\ref{main}. For that, first note that for our solution given in Equation~\nn{Maxsol}, 
$$
i\cdot j_a=8\varphi_{na}\:\implies\:
i \nabla_a j_b = -8 \varphi_{ab}\, ,
$$
and also
$$
i\nabla_c F_{ab}=2n_c \varphi_{ab}+\delta_{ca}\varphi_{nb}
-\delta_{cb}\varphi_{na}
+n_a \varphi_{cb}-n_b \varphi_{ca}\, .
$$
Thus Equation~\nn{Eren} gives
\begin{align*}
{\mathcal E}_{\rm ren}&=
64 \varphi^{na}\varphi_{na}
+\tfrac12[6(1-2\sigma)|\varphi|^2_\delta
+28 \varphi^{na}\varphi_{na}]
+3[-8(1+2\sigma)|\varphi|_\delta^2+16 \varphi^{na}\varphi_{na}]
\\
&
=
-(21+54\sigma)|\varphi|_\delta^2
+126 \varphi^{na}\varphi_{na}\, ,
\end{align*}
so that integrating the above over the unit six ball and applying Equation~\nn{Ipleadthesixth}, we have 
$$
\frac14
\int_B  \ext {\rm Vol}(\delta)
{\mathcal E}_{\rm ren}=
-\frac{37}{32} \hh {\rm Vol}(S^5)|\varphi|_\delta^2
+\frac{63}{16}\hh    \int_{S^5} 
 \ext {\rm Vol}(S^5) \varphi^{\hat na}\varphi_{\hat na}
 =-
\frac12 {\rm Vol}(S^5)|\varphi|_\delta^2\, .
$$
This precisely matches the explicit computation performed above.

\color{black}

\section*{Acknowledgements}
We  thank Jeffrey Case 
for explaining Remark~\ref{case} to us 
and Bruno Nachtergaele  for pointing out the treatment of Schr\"odinger operators in References~\cite{Reed,Kato}. We also thank Jesse Gell-Redman and Fang Wang  for useful discussions. In addition, we are very grateful to Rafe Mazzeo for assisting us with the proof of Theorem~\ref{mainth}.
A.R.G.'s contribution to this work was partially supported by a grant from the Simons Foundation during a visit to the 
 Isaac Newton Institute 
for Mathematical Sciences.
A.R.G. and A.W. acknowledge support from the Royal Society of New Zealand via Marsden Grant 19-UOA-008.  A.W. was also supported by Simons Foundation Collaboration Grant for Mathematicians ID 686131. This article is based upon work supported by the European Cooperation in
Science and Technology  Action 21109 CaLISTA, {\tt www.cost.eu}, HORIZON-MSCA-2022-SE-01-01 CaLIGOLA, MSCA-DN CaLiForNIA - 101119552. Y.Z. acknowledges support from
National Key Research and Development Project SQ2020Y\-FA070080 and NSFC 12071450.

\end{document}